\documentclass[10pt]{article}
\usepackage{amssymb,amsthm}
\theoremstyle{plain}
\newtheorem{thm}{Theorem}
\theoremstyle{definition}

\font\smallrm=cmr8
\newcommand{\N}{\mathbb N}
\newcommand{\R}{\mathbb R}
\newcommand{\Q}{\mathbb Q}
\title{Applications of the Hales-Jewett\\ Theorem near zero}

\date{}
\author{Pintu Debnath
\footnote{Department of Mathematics, 
Basirhat College, Basirhat-743412, North
24th Parganas, West Bengal, India
{\tt pintumath1989@gmail.com}}
\and
Sayan Goswami
\footnote{Department of Mathematics, 
University of Kalyani, Kalyani-741235,
Nadia, West Bengal, India {\tt sayan92m@gmail.com}\hfill\break
}}

\begin{document}

\maketitle

\begin{abstract}
 The famous van der Waerden's theorem states that if $\mathbb{N}$
is finitely colored then one color class will contain arithmetic progressions
of arbitrary length. The polynomial van der Waerden's theorem says
that if $p_{1}(x),p_2(x),\ldots,p_{k}(x)$ are polynomials with integer coefficients
and zero constant term and $\mathbb{N}$ is finitely colored,
then there exist $a,d\in\mathbb{N}$ such that $\big\{a+p_{t}(d):t\in\{1,2,\ldots,k\} \big\} $
is monochromatic. There are dynamical, algebraic, and combinatorial
proofs of this theorem. In this article we will prove a ``near zero" version
of the polynomial van der Waerden's theorem. That is, we will show
that for any $\epsilon>0$, if $(0,\epsilon)\cap\mathbb{Q}$
is finitely colored then there exists $a,d\in\mathbb{Q}\setminus\{0\}$
such that $\big \{a+p_{t}(d):t\in\{0,1,2,\ldots,k\}\big\}$
is monochromatic. 
\end{abstract}

\section{Introduction}

N. Hindman and I. Leader in \cite{HL} first introduced the
study of the semigroup of ultrafilters near 0. This study deals
with partition results using the algebra of the  Stone-\v{C}ech compactification
of subsemigroups of $(\R,+)$ viewed as discrete spaces.
They developed a central sets theorem near zero and, as a corollary, 
proved a version of van der Waerden's 
theorem near zero. 

Although there are dynamical \cite{BL},
algebraic \cite{H} and combinatorial \cite{W} proofs of the
polynomial van der Waerden's theorem over $\mathbb{Z}$, 
there has not been a near zero version.

We point out first that a near zero version of 
van der Waerden's theorem is an immediate consequence
of the finitistic version of van der Waerden's theorem 
itself \cite{vdW}, which says that whenever $r,k\in\mathbb{N}$
there exists $n\in\N$ such that whenever
$\{1,2,\ldots,n\}$ is $r$-colored, there exists $a,d\in
\{1,2,\ldots,n\}$ such that 
$\{a,a+d,a+2d,\ldots,a+kd\}$ is monochromatic. (See \cite[page 9]{GRS}.)

\begin{thm}\label{vdwnearzero} Let $0<\varepsilon<1$ be given.
For any $r,k\in\mathbb{N}$ there exists a finite
set $A\subseteq(0,\varepsilon)\cap\mathbb{Q}$ such that whenever $A$ is partitioned
into $r$ cells, there will exist $a,d\in(0,\varepsilon)\cap\mathbb{Q}$
such that $\{a,a+d,a+2d,\ldots,a+kd\}$ is monochromatic. 
\end{thm}

\begin{proof} Let $r,k\in\mathbb{N}$ be given. 
Pick $n$ as guaranteed by the finitistic version of
van der Waerden's theorem for $r$ and $k$, pick
$M\in\N$ such that $M>\frac{n}{\varepsilon}$, and let
$A=\big\{{t\over M}:t\in\{1,2,\ldots,n\}\big\}$.
If $A=\bigcup_{i=1}^rC_i$ and for 
$i\in\{1,2,\ldots,r\}$, $D_i=\{M\cdot t:t\in C_i\}$,
then $\{1,2,\ldots,n\}=\bigcup_{i=1}^r D_i$. If $i\in\{1,2,\ldots,r\}$
and $\{a,a+d,\ldots,a+kd\}\subseteq D_i$, then
$\{\frac{a}{M},\frac{a}{M}+\frac{d}{M},\ldots,\frac{a}{M}+\frac{kd}{M}\}\subseteq D_i$
\end{proof}

Theorem \ref{vdwnearzero} is not as strong as
\cite[Corollary 5.1]{HL}, but it has the advantage
that it depends only on the finitistic version of
van der Waerden's theorem, which has an elementary
proof, while \cite[Corollary 5.1]{HL} uses the algebra
of the Stone-\v Cech compactification, so needs the
Axiom of Choice.

The reason the proof of Theorem \ref{vdwnearzero} is
so trivial is that arithmetic progressions are 
linear objects.  In the next section we will
establish two nonlinear results near zero. In both of
these, the proofs are elementary. They do not utilize the
algebraic structure of the Stone-\v Cech compactification
of a discrete semigroup.

The first of
these results is the following which establishes the existence of geo-arithmetic
progression near $0$.

\begin{thm} \label{geoarith} Let $0<\epsilon<$1 and $k,r\in\mathbb{N}$ be given.
For any $r$-coloring of $(0,\epsilon)\cap\mathbb{Q}$ there exist $a$ and 
$d$ in $(0,\epsilon)\cap\mathbb{Q}$ such that
$\big\{b\cdot(a+i\cdot d)^{j}:i,j\in\{0,1,\ldots,k\}\big\}$ is monochromatic.
\end{thm}

Theorem \ref{geoarith} is a very special case of \cite[Theorem 2.9]{DP}
which was derived using the algebraic structure of $\beta S$.
We prove our theorem using Theorem \ref{thmBM}, which is
\cite[Theorem 1.5]{BM} and has an elementary proof.

Our major result is the following, which we believe is new. It
will be proved using Theorem \ref{Walters}, which is the main
theorem of \cite{W} and has an elementary proof.

\begin{thm}\label{polynearzero} Let $n\in \N$ and let $\big\{P_{i}:i\in\{1,2,\ldots,n\}\big\}$ be a collection
of polynomials on $\mathbb{Q}$ with $0$ constant term. Then for
every $0<\varepsilon<1$ and $k\in\mathbb{N}$, if we partition
$(0,\varepsilon)\cap\mathbb{Q}$ into $k$ cells, then one of them
contains a configuration of the form $\{a,a+P_{1}(\alpha),a+P_{2}(\alpha),\ldots,a+P_{n}(\alpha)\}$
where $a,\alpha\in\mathbb{Q}$. 
\end{thm}

\section{Progressions near zero}

Let $\mathcal{F}$ denote the set of all finite nonempty subsets of
$\mathbb{N}$. A subfamily $\mathcal{P}$ of $\mathcal{F}$ is
said to be {\it partition regular\/} if for any finite partition of $\mathbb{N}$
there is a monochromatic member of $\mathcal{P}$.
Given $\alpha$ and $\beta$ in ${\mathcal F}$, we write $\alpha<\beta$ when
$\max \alpha<\min\beta$.  If $k,N\in\N$,  $w=u_1u_2\cdots u_N \in\{0,1,\ldots,k\}^N$,
$\emptyset\neq\alpha\subseteq \{0,1,\ldots,k\}$, $u_i=0$ for each $i\in\alpha$,
and $t\in\{0,1,\ldots,k\}$, then $w^\alpha(t)$ is the 
word obtained from $w$ by replacing $u_i$ by $t$ for each $i\in\alpha$.

Theorems \ref{thmBM} and \ref{Walters} are extensions of the original 
Hales-Jewett theorem \cite{HJ}.

\begin{thm}\label{thmBM} Let $k,r\in\mathbb{N}$ and suppose $\mathcal{P}$
is a partition regular family of finite subsets of $\mathbb{N}$.
Then there exists $N=N(k,r,\mathcal{P})\in\mathbb{N}$ such that if
$\{0,1,\ldots,k\}^{N}$ is $r$-colored, then there exist $l\in\mathbb{N}$,
$\beta_{i}\in\mathcal{P}$, $1\leq i\leq l$,
with $\beta_{1}<\beta_{2}<\ldots<\beta_{l}<\{N+1\}$, and $w=u_{1}u_{2}\ldots u_{N}\in\{0,1,\ldots,k\}^{N}$
having the property that $u_{i}=0$ for all $i\in\bigcup_{j=1}^{l}\beta_{j}$,
such that $\{ w^{\{j_{1},j_{2},\ldots,j_{l}\}}(t):j_{i}\in\beta_{i}\,,\,1\leq i\leq l\,,\,
0\leq t\leq k\}$ is monochromatic.
\end{thm}

\begin{proof} \cite[Theorem 1.5]{BM}. \end{proof}

For $q\in\N$, we let $[q]=\{1,2,\ldots,q\}$.

\begin{proof}[Proof of Theorem {\rm \ref{geoarith}}.]
  The proof is similar to a proof in \cite{B} 

Let ${\mathcal P}$ be the set of
length $k+1$ arithmetic progressions in $\N$. By van der Waerden's
Theorem, ${\mathcal P}$ is partition regular. Let $N=N(k,r,{\mathcal P})$ be
as guaranteed by Theorem \ref{thmBM}. Thus if $\{0,1,\ldots ,k\}^N$ is $r$-colored, then
there exist $l\in\N$ and $\beta_1,\beta_2,\ldots,\beta_l\in{\mathcal P}$
such that $\beta_1<\beta_2<\ldots<\beta_l<\{N+1\}$ and there
exists $w=u_1u_2\cdots u_N\in\{0,1,\ldots ,k\}^N$ such that
$u_i=0$ for each $i\in\bigcup_{j=1}^l\beta_j$ and
$$\{w^{\{j_1,j_2,\ldots,j_l\}}(x):\hbox{each }j_i\in\beta_i\,,\,1\leq i\leq l\,,\,\hbox{ and }
0\leq x\leq k\}$$ is monochromatic.

Pick $P,M\in\N$ such that ${1\over P}<\epsilon$ and ${N\over M}<\epsilon$
and define $f:\{0,1,\ldots ,k\}^N\to(0,\epsilon)\cap\Q$ by 
$f(\alpha)={1\over P}\cdot\prod_{t\in[N]}\left({t\over M}\right)^{\alpha(t)}$.
Let $\varphi:(0,\epsilon)\cap\Q\to\{1,2,\ldots,r\}$ and let $\psi=\varphi\circ f$.  Pick
$l,\beta_1,\ldots,\beta_l$ and $w=u_1u_2\cdots u_N$ such that
$\psi$ is constant on 
$$\{w^{\{j_1,j_2,\ldots,j_l\}}(x):\hbox{each }j_i\in\beta_i\,,\,1\leq i\leq l\,,\,\hbox{ and }
0\leq x\leq k\}\,.$$
For each $i\in\{1,2,\ldots,l\}$, pick $a_i$ and $b_i$ in $\N$ such that
$\beta_i=\big\{a_i+j\cdot b_i:j\in\{0,1,\ldots ,k\}\big\}$.
For $j,q\in\{0,1,\ldots ,k\}$, let $\alpha_{j,q}=
w^{\{a_1+jb_1,a_2,\ldots,a_l\}}(q)$. Then for 
$t\in\{1,2,\ldots,N\}$, 
$$\alpha_{j,q}(t)=\left\{\begin{array}{rl}u_t&\hbox{if }t\notin \{a_1+jb_1,a_2,\ldots,a_l\}\\
q&\hbox{if }t\in\{a_1+jb_1,a_2,\ldots,a_l\}\,.\end{array}\right.$$
Then $\varphi$ is constant on $\big\{f(\alpha_{j,q}):j,q\in\{0,1,\ldots ,k\}\big\}$.

Let $C=[N]\setminus\bigcup_{i=1}^l\beta_i$. Using the fact that $u_t=0$ if $t\in\bigcup_{i=1}^l\beta_i$, we have that 
$$\textstyle f(\alpha_{j,q})={1\over P}\cdot\prod_{t\in C}\left({t\over M}\right)^{u_t}
\cdot\left({a_1+jb_1\over M}\right)^q\cdot\prod_{i=2}^l\left({a_i\over M}\right)^q\,.$$
Let $B={1\over P}\cdot\prod_{t\in C}\left({t\over M}\right)^{u_t}$, let
$A=\left({a_1\over M}\right)\cdot \prod_{i=2}^l\left({a_i\over M}\right)$ and let
$D=\left({b_1\over M}\right)\cdot \prod_{i=2}^l\left({a_i\over M}\right)$.
Then for $j,q\in\{0,1,\ldots ,k\}$, $f(\alpha_{j,q})=B\cdot(A+j\cdot D)^q$.\end{proof}

Theorem \ref{Walters} uses some special notation, which we introduce now.

For $q,N\in\N$, $Q=[q]^{N}$, $\emptyset\neq \gamma\subseteq[N]$ and $1\leq x\leq q$,
$a\oplus x\gamma$ is defined to be the vector $b$ in $Q$ obtained
by setting $b_{i}=x$ if $i\in\gamma$ and $b_{i}=a_{i}$ otherwise.

In the statement of Theorem \ref{Walters}, we have $a\in Q$ so that $a=\langle \vec a_1,\vec a_2,
\ldots,\vec a_d\rangle$ where for $j\in\{1,2,\ldots d\}$, 
$\vec a_j\in [q]^{N^j}$ and we have $\gamma\subseteq [N]=\{1,2,\ldots,N\}$.
Given $j\in\{1,2,\ldots,d\}$, let $\vec a_j=\langle a_{j,\vec i}\rangle_{\vec i\in N^j}$.
Then $a\oplus x_1\gamma\oplus x_2(\gamma\times\gamma)\oplus\ldots\oplus x_d\gamma^d=b$
where $b=\langle \vec b_1,\vec b_2,
\ldots,\vec b_d\rangle$ and for $j\in\{1,2,\ldots,d\}$, 
$\vec b_j=\langle b_{j,\vec i}\rangle_{\vec i\in N^i}$
where $$b_{j,\vec i}=\left\{\begin{array}{rl}x_j&\hbox{if }
\vec i\in \gamma^i\\
a_{j,\vec i}&\hbox{otherwise.}\end{array}\right.$$

\begin{thm}
 \label{Walters} For any $q,k,d$ there exists $N$ such
that whenever $Q=Q(N)=[q]^{N}\times[q]^{N\times N}\times\ldots\times[q]^{N^{d}}$
is $k-colored$ there exists $a\in Q$ and $\gamma\subset[N]$ such
that the set of points $\{a\oplus x_{1}\gamma\oplus x_{2}(\gamma\times\gamma)\oplus\ldots\oplus x_{d}\gamma^{d}:1\leq x_{i}\leq q\}$
is monochromatic. 
\end{thm}

\begin{proof} \cite[Polynomial Hales-Jewett Theorem]{W}. \end{proof}

We now present the proof of our main theorem.

\begin{proof}[Proof of Theorem {\rm \ref{polynearzero}}]
  This proof is based on the derivation of the Polynomial
van der Waerden Theorem from the Polynomial Hales-Jewett Theorem in
\cite{W}

For $i\in\{1,2,\ldots,n\}$ let $d_i=\deg(P_i)$ and let $d=\max\big\{d_i:i\in\{1,2,\ldots,n\}\big\}$.
Let $\langle a^i_j\rangle_{j=1}^{d_i}$ be the coefficients of $P_i$ and for
$d_i<j\leq d$ (if any) let $a^i_j=0$ so that $P_i(x)=\sum_{j=1}^d a^i_jx^j$.
Let $q=n\cdot d$ and let $N$ be as guaranteed by Theorem 6 for $k$, $q$, and $d$.
Let $m=\max\big\{|a^i_j|:i\in\{1,2,\ldots,n\}\hbox{ and }j\in\{1,2,\ldots,d\}\big\}$ and pick $b\in\N$
such that $(m\cdot\sum_{j=1}^d N^j)/b<{\epsilon\over 4}$.

Let $A=\big\{{a^i_j\over b^j}:i\in\{1,2,\ldots,n\}\hbox{ and }j\in\{1,2,\ldots,d\}\big\}$.
As noted by Walters in [W], it is the cardinality of $[q]$ in Theorem 6 that
matters, and $|A|\leq q$ so if $Q=A^N\times A^{N^2}\times\ldots\times A^{N^d}$ and
$Q$ is $k$-colored, then
there exist $u\in Q$ and $\gamma\subseteq A$ such that
$\{u\oplus x_1\gamma\oplus x_2(\gamma\times\gamma)\oplus\ldots\oplus x_d\gamma^d:\hbox{each }x_j\in A\}$
is monochromatic.

Pick $r\in({\epsilon\over 4},{\epsilon\over 2})\cap\Q$ and define
$\sigma:Q\to(0,\epsilon)\cap\Q$ as follows. Let $u\in Q$, where
$u=\langle \vec u_1,\vec u_2,\ldots,\vec u_d\rangle$ and for
$j\in\{1,2,\ldots,d\}$, $\vec u_j=\langle u_{j,\vec i}\rangle_{\vec i\in N^j}$.
Then $$\textstyle\sigma(u)=r+\sum_{j=1}^d\sum_{\vec i\in N^j}u_{j,\vec i}\,.$$
If $x\in A$, then $|x|\leq {m\over b}$ so
$$\textstyle|\sum_{j=1}^d\sum_{\vec i\in N^j}u_{j,\vec i}|
\leq \sum_{j=1}^d\sum_{\vec i\in N^j}|u_{j,\vec i}|
\leq\sum_{j=1}^d\sum_{\vec i\in N^j}{m\over b}={m\over b}\cdot\sum_{j=1}^dN^j<{\epsilon\over 4}$$
so $\sigma(u)\in (0,\epsilon)$.

Let $\varphi:(0,\epsilon)\cap\Q\to\{1,2,\ldots,k\}$.
Then $\varphi\circ\sigma$ is a $k$-coloring of $Q$. Pick 
$u\in Q$ and $\gamma\subseteq A$ such that $\varphi\circ\sigma$ is constant on 
$\{u\oplus x_1\gamma\oplus x_2(\gamma\times\gamma)\oplus\ldots\oplus x_d\gamma^d:\hbox{each }x_j\in A\}$.
Let $s=\sum_{j=1}^d\sum_{\vec i\in N^j\setminus\gamma^j}u_{j,\vec i}$ and let $c=|\gamma|$.
If $x_j\in A$ for each $j\in\{1,2,\ldots,d\}$,
$\sigma(u\oplus x_1\gamma\oplus x_2(\gamma\times\gamma)\oplus\ldots\oplus x_d\gamma^d)=
r+s+\sum_{j=1}^d x_jc^j$. In particular, if $i\in\{1,2,\ldots,n\}$, $j\in\{1,2,\ldots,d\}$ and
$x_j={a^i_j\over b^j}$, then\hfill\break 
$\sigma(u\oplus x_1\gamma\oplus x_2(\gamma\times\gamma)\oplus\ldots\oplus x_d\gamma^d)=
r+s+\sum_{j=1}^d a^i_j\cdot\left({c\over b}\right)^j=r+s+P_i\left({c\over b}\right)$.\end{proof}

\vspace{2 cm}
\noindent\textbf{Acknowledgment.} The second author of the paper acknowledges
the grant UGC-NET SRF fellowship with id no. 421333 of CSIR-UGC NET.
 We acknowledge the anonymous referee for several helpful
comments on the paper.

\bibliographystyle{plain}

\end{document}